\documentclass[12pt]{amsart}
\usepackage{enumerate}
\usepackage{amssymb}
\usepackage{amsmath}
\usepackage{amsthm}
\newcommand{\snug}{\unskip\kern-\mathsurround}
\newcommand{\ad}{{\rm ad}}

\newcommand{\gr}{{\rm gr}}

\newcommand{\Z}{{\mathbb Z}}
\newcommand{\F}{{\mathbb F}}

\newcommand{\Q}{{\mathbb Q}}

\newcommand{\tr}{{\rm tr}}
\newtheorem{theorem}{Theorem}[section]
\newtheorem{lemma}[theorem]{Lemma}

\newtheorem{corollary}[theorem]{Corollary}
\newtheorem{conjecture}[theorem]{Conjecture}

\theoremstyle{definition}
\newtheorem{definition}[theorem]{Definition}

\begin{document}
\title{Linking Numbers and the Tame Fontaine-Mazur Conjecture}
\author{John Labute}
\address{Department of Mathematics and Statistics, McGill University, Burnside Hall, 805 Sherbrooke Street West, Montreal QC H3A 0B9, Canada}
\email{labute@math.mcgill.ca}
\begin{abstract}
Let $p$ be an odd prime, let $S$ be a finite set of primes $q\equiv1\ {\rm mod}\ p$ but $q\not\equiv1\ {\rm mod}\ p^2$  and let $G_S$ be the Galois group of the maximal $p$-extension of $\Q$ unramified outside of $S$. If $\rho$ is a continuous homomorphism of $G_S$ into ${\rm GL}_2(\Z_p)$ then under certain conditions on the linking numbers of $S$ we show that $\rho=1$ if $\overline\rho=1$. We also show that $\overline\rho=1$ if $\rho$ can be put in triangular form mod $p^3$.
\end{abstract}
\date{May 15, 2013}
\thanks{$^\dag$Research supported in part by an NSERC Discovery Grant.}
\subjclass{11R34, 20E15, 12G10, 20F05, 20F14, 20F40}

\maketitle
\textbf{}\hfill{\it To Helmut Koch on his 80th birthday}
\section{Statement of Results} Let $p$ be a rational prime. Let $K$ be a number field, let $S$ be a finite set of primes of $K$ with residual characteristics $\ne p$ and let $\Gamma_{S,K}$ be the Galois group of the maximal (algebraic) extension of $K$ unramified outside of $S$. The Tame Fontaine-Mazur Conjecture (cf. \cite{FM}, Conj. 5a) states that every continuous homomorphism
$$
\rho : \Gamma_{S,K}\rightarrow{\rm GL}_n(\Z_p)
$$
has a finite image. If $\overline\rho$ is the reduction of $\rho$ mod $p$ then $\overline\rho$ is trivial if and only if the image of $\rho$ is contained in the standard subgroup
$$
{\rm GL}_n^{(1)}(\Z_p)=\{X\in{\rm GL}_n(\Z_p)\mid X\equiv1\  {\rm mod}\  p\}
$$
which is a pro-$p$-group. Hence, if $\overline\rho=1$, the homomorphism $\rho$ factors through $G_{S,K}$, the maximal pro-$p$-quotient of $\Gamma_{S,K}$. Since ${\rm GL}_n^{(1)}(\Z_p)$ is torsion free, this shows that when $\overline\rho=1$ the Tame Fontaine-Mazur Conjecture is equivalent to the following conjecture.
\medskip

\begin{conjecture}~\label{WTFM}
If $\rho:G_{S,K} \rightarrow{\rm GL}_n^{(1)}(\Z_p)$ is a continuous homomorphism then $\rho=1$.
\end{conjecture}

Conversely, the truth of Conjecture~\ref{WTFM} for any number field $K$ implies the Fontaine-Mazur Conjecture. In this paper we will prove Conjecture~\ref{WTFM} when $K=\Q$ for certain sets $S$.

We now let $K=\Q$ and $G_S=G_{S,\Q}$. To prove Conjecture~\ref{WTFM} we can assume that the primes in $S$ are congruent to $1$ mod $p$ since these are the only primes different from $p$ that can ramify in a $p$-extension of $\Q$. We will also assume that the primes in $S$ are not congruent to $1$ mod $p^2$, which is equivalent to $G_S/[G_S,G_S]$ being elementary. In this case we  will show that Conjecture~\ref{WTFM} follows from a Lie theoretic analogue of it when $p$ is odd. We therefore assume that $p\ne2$ for the rest of the paper.

To formulate this analogue let $S=\{q_1,\ldots,q_d\}$ and let $\mathfrak l_S$ be the finitely presented Lie algebra over $\F_p$ generated by $\xi_1,\ldots,\xi_d$ with relators $\sigma_1,\ldots,\sigma_d$ where
$$
\sigma_i=c_i\xi_i+\sum_{j\ne i}\ell_{ij}[\xi_i,\xi_j]
$$
with $c_i=(q_i-1)/p$ mod $p$ and the linking number $\ell_{ij}$ of $(q_i,q_j)$ defined by $q_i \equiv g_j^{-\ell_{ij}}$ mod $q_j$ with $g_j$ a primitive root mod~$q_j$. We call $\mathfrak l_S$ the linking algebra of $S$. Up to isomorphism, it is independent of the choice of primitive roots.
\medskip

\begin{theorem}\label{main}
There exists a mapping
$$
\ell:{\rm Hom}_{cont}(G_S,{\rm GL}_n^{(1)}(\Z_p))\rightarrow{\rm Hom}(\mathfrak l_S,gl_n(\F_p))
$$
such that $\rho=1\iff\ell(\rho)=0$.
\end{theorem}

\begin{corollary}
If the cup-product $H^1(G_S,\F_p)\times H^1(G_S,\F_p)\rightarrow H^2(G_S,\F_p)$ is trivial then Conjecture~\ref{WTFM} is true for $G_S$.
\end{corollary}

\begin{definition}[Property $FM(n)$]
A lie algebra $\mathfrak g$ over a field $F$ is said to have Property $FM(n)$ if every $n$-dimensional representation of $\mathfrak g$ is trivial.
\end{definition}
\medskip

\begin{theorem}~\label{main2}
If $\mathfrak l_S$ has Property $FM(k)$ then Conjecture~\ref{WTFM} is true for $n=k$.
\end{theorem}
\medskip

If $|S|\le2$ then $\mathfrak l_S$ has Property $FM(n)$ for all $n$ since $\mathfrak l_S=0$ in this case. However $\mathfrak l_S$ may not have Property $FM(2)$ if $|S|\ge 3$; for example, if $p=3$, $S=\{7,31,229\}$ or if $p=5$ and $S=\{11,31,1021\}$. However, the number of such $S$ is relatively small; for example, if $p=7$ and the primes in $S$ are at most $10,000$, the set $S$ fails to have Property $FM(2)$ approximately $.2\%$ of the time. The following theorem gives necessary and sufficient conditions for Property $FM(n)$ to hold when $|S|=3$.
\begin{theorem}\label{S3}
Let $m_{ij}=-\ell_{ij}/c_i$. If $|S|=3$ and $n<p$ then Property $FM(n)$ holds if and only if one of the following conditions holds:
\begin{enumerate}[{\rm (a)}]
\item $m_{ij}=0$ for some $i,j$;
\item $m_{ij}\ne 0$ for all $i,j$ and $m_{ik}=m_{jk}$ for some $i,j,k$ with $i\ne j$;
\item $m_{ij}\ne0$ for all $i,j$ and $(m_{ik}-m_{jk})(m_{ki}m_{ij}-m_{kj}m_{ji})\ne0$ for some $i,j,k$.
\end{enumerate}
These conditions are independent of the choice of primitive roots.
\end{theorem}
\begin{theorem}
If $|S|=3$ and $n<p$ then $\mathfrak l_S$ fails to have Property $FM(n)$ if and only if  $\ell_{ij}\ne0$ for all $i,j$ and $ \ell_{13}/c_1=-\ell_{23}/c_2,\ \ell_{21}/c_2=-\ell_{31}/c_3,\ \ell_{12}/c_1=-\ell_{32}/c_3$.
\end{theorem}

\begin{theorem}\label{FM2}
Let $\rho:G_S\rightarrow{\rm GL}_2(\Z_p)$ be a continuous homomorphism. Then $\overline\rho=1$ if $\rho$ can be brought to triangular form mod $p^3$.
\end{theorem}

The pro-$p$-groups $G_S$ are very mysterious. They are all fab groups, i.e., subgroups of finite index have finite abelianizations, and for $|S|\ge4$ they are not $p$-adic analytic. So far no one has given a purely algebraic construction of such a pro-$p$-group. We call a pro-$p$-group $G$ a Fontaine-Mazur group if every continuous homomorphism of $G$ into ${\rm GL}_n(\Z_p)$ is finite. Again, no purely algebraic construction of such a group exists. In this direction we have the following result.

\begin{theorem}\label{wfmg}
Let $G$ be the pro-$p$-group with generators $x_1,\ldots,x_{2m}$ and relations
$$
x_1^{pc_1}[x_1,x_2]=1,\ x_2^{pc_2}[x_2,x_3]=1,\ldots,\ x_{2m-1}^{pc_{2m-1}}[x_{2m-1,2m}]=1,\ x_{2m}^{pc_{2m}}[x_{2m},x_1]=1
$$
with $c_i\not\equiv 0$ mod $p>2$ and $m\ge2$. Then every continuous homomorphism of $G$  into ${\rm GL}_n^{(1)}(\Z_p)$ is trivial if $n<p$.
\end{theorem}

\section{Mild pro-$p$-groups}
Let $G$ be a pro-$p$-group. The descending central series of $G$ is the sequence of subgroups $G_n$ defined for $n\ge1$ by
$$
G_1=G,\quad G_{n+1}=G_n^p[G,G_n]
$$
where $G_n^p[G,G_n]$ is the closed subgroup of $G$ generated by $p$-th powers of elements of $G_n$ and commutators of the form $[h,k]=h^{-1}k^{-1}hk$ with $h\in G$ and $k\in G_n$. The graded abelian group
$$
\gr(G)=\oplus_{n\ge1}\gr_n(G)=\oplus_{n\ge1}G_n/G_{n+1}
$$
is a graded vector space over $\F_p$ where $\gr_n(G)$ is denoted additively. We let $$
\iota_n :G_n\rightarrow\gr_n(G)
$$
be the quotient map. Since $p\ne2$, the graded vector space $\gr(G)$ has the structure of a graded Lie algebra over $\F_p[\pi]$ where
$$
\pi\,\iota_n(x)=\iota_{n+1}(x^p),\quad [\iota_n(x),\iota_m(y)]=\iota_{n+m}([x,y]).
$$

Let $G=F/R$ where $F$ is the free pro-$p$-group on $x_1,\ldots,x_d$ and $R=(r_1,\ldots,r_m)$ is the closed normal subgroup of $F$ generated by $r_1,\ldots,r_m$ with $r_i\in F_2$. If
$$
r_k\equiv\prod_{i\ge1}x_i^{pa_j}\prod_{i<j}[x_i,x_j]^{a_{ijk}}\ \ {\rm mod}\ F_3
$$
and we let $\xi_i=\iota_1(x_1)$, $\rho_k=\iota_2(r_k)$ in $L=\gr(F)$ then $L$ is the free Lie algebra over $\F_p[\pi]$ on $\xi_1,\ldots,\xi_d$ and
$$
\rho_k=\sum_{i\ge1}a_i\pi\xi_i+\sum_{i<j}a_{ijk}[\xi_i,\xi_j].
$$
Let $\mathfrak r$ be the ideal of $L$ generated by $\rho_1,\ldots\rho_m$, let $\mathfrak g=L/\mathfrak r$ and let $U$ be the enveloping algebra of $\mathfrak g$. Then $\mathfrak r/[\mathfrak r,\mathfrak r]$ is a $U$-module via the adjoint representation. The sequence $\rho_1,\ldots,\rho_m$ is said to be {\bf strongly free} if (a) $\mathfrak g$ is a torsion-free $\F_p[\pi]$-module and (b) $\mathfrak r/[\mathfrak r,\mathfrak r]$ is a free $U$-module on the images of $\rho_1,\ldots,\rho_m$ in which case we say that the presentation is strongly free.
\begin{theorem}[\cite{La}, Theorem 1.1]
If $G=F/R$ is strongly free then $\mathfrak r$ is the kernel of the canonical surjection $\gr(F)\rightarrow\gr(G)$ so that $\gr(G)=L/\mathfrak r$.
\end{theorem}
A finitely presented pro-$p$-group $G$ is said to be {\bf mild} if it has a strongly free presentation.

Let $A=\Z_p[[G]]$ be the completed algebra of $G$ and let $I={\rm Ker}( A\rightarrow\F_p)$ be the augmentation ideal of $\Z_p[[G]]$. Then
$$
\gr(A)=\oplus_{n\ge1} I^n/I^{n+1}
$$
is a graded algebra over $\F_p[\pi]$ where $\pi$ can be identified with the image of $p$ in $I/I^2$. The canonical injection of $G$ into $A$ sends $G_n$ into $1+I^n$ and gives rise to a canonical Lie algebra homomorphism of $\gr(G)$ into $\gr(A)$ which is injective if and only if $G_n=G\cap(1+I^n)$.
\begin{theorem}[\cite{La}, Theorem 1.1]
If $G$ is mild the canonical map $\gr(G)\rightarrow\gr(A)$ is injective and $\gr(A)$ is the enveloping algebra of $\gr(G)$.  Moreover, $R/[R,R]$ is a free $A$-module which implies that ${\rm cd}(G)\le2$.
\end{theorem}

We now give a criterion for the mildness of $G=G_S$ when $p\ne 2$ and $p\notin S$. The group $G_S$ has a presentation $F(x_1,\ldots,x_d)/(r_1,\ldots,r_d)$
where $x_i$ is a lifting of a generator of an inertia group at $q_i$ and
$$
r_i=x_i^{pc_i}\prod_{j\ne i}[x_i,x_j]^{\ell_{ij}}\ \ {\rm mod}\ F_3
$$
which is due to Helmut Koch~(\cite{Ko}, Example 11.11). Using the transpose of the inverse of the transgression isomorphism
$$
{\rm tg}:H^1(R,\F_p)^F=(R/R^p[R,F])^*\longrightarrow H^2(G,\F_p),
$$
the relator $r_i$ defines a linear form $\phi_i$ on $H^2(G,\F_p))$ such that, if $\chi_1,\ldots,\chi_d$ is the basis of $H^1(F,\F_p)=(F/F^p[F,F])^*$ with $\chi_i(x_j)=\delta_{ij}$, we have $\phi_i(\chi_i\cup\chi_j)=-\ell_{ij}$ if $i<j$; cf.\cite{Ko}, Theorem 7.23.

The set $S$ is said to be a {\bf circular set} of primes if there is an ordering $q_1,\ldots,q_d$ of the set $S$ such that

(a) $\ell_{i,i+1}\ne0$ for $1\le i<d$ and $\ell_{d1}\ne 0$,

(b) $\ell_{ij}=0$ if $i,j$ are odd,

(c) $\ell_{12}\ell_{23}\cdots\ell_{d-1,d}\ell_{d1}\ne\ell_{1m}\ell_{m,m-1}\cdots \ell_{32}\ell_{21}$.
\begin{theorem}\label{gsm}
If $S$ is a circular set of primes then $G_S$ is mild.
\end{theorem}
\begin{theorem}\label{cs1}
The set $S$ can be extended to a set $S\cup{q}$ where $q\equiv1$ mod $p$, $q\not\equiv1$ mod $p^2$ in such a way that the pairs $(q,q_i)$, $(q_i,q)$ with non-zero linking numbers can be arbitrarily prescribed.
\end{theorem}
\medskip

\begin{corollary}~\label{mex}
The set $S$ can always extended to a set $S'$ with $G_{S'}$ mild.
\end{corollary}
See Labute~(\cite{La}, Theorem~1.1) for the proof of Theorem~\ref{gsm} and (\cite{La}, Proposition 6.1) for the proof of Theorem~\ref{cs1}. The proof of Proposition~6.1 in \cite{La} yields the sharper form stated here.
\medskip

\begin{theorem}~\label{Lmex}
There exists a finite set $S'\supseteqq S$ consisting of primes $q\equiv1$ mod $p$, $q\not\equiv 1$ mod $p^2$ such that $G_{S'}$ is mild and, if $n< p$, the Lie algebra $\mathfrak l_{S'}$ has Property(FM(n)) if $\mathfrak l_S$ does.
\end{theorem}

\section{Proof of Theorem~\ref{main}}
Let $G$ be a pro-$p$-group with $G/[G,G]\cong(\Z/p\Z)^d$ and let $\rho : G\rightarrow {\rm GL}_n^{(1)}(\Z_p)$ be a continuous homomorphism.  Let
$$
{\rm GL}_n^{(k)}(\Z_p)=\{X\in{\rm GL}_n(\Z_p)\mid X\equiv1\  {\rm mod}\  p^k\}.
$$
\begin{lemma}
Let $X=1+p^iA\in{\rm GL}_n^{(i)}(\Z_p)$, $Y=1+p^jB\in{\rm GL}_n^{(j)}(\Z_p)$ then
$$
[X,Y]=1+p^{i+j}[A,B]\ {\rm mod}\ p^{i+j+1},\ X^p=1+p^{i+1}A\ {\rm mod}\ p^{i+2},
where\ [A,B]=AB-BA.
$$
\end{lemma}
\medskip

\begin{lemma}~\label{H}
If $\rho(G)\ne1$ then $\rho(G)\not\subseteq {\rm GL}_n^{(2)}(\Z_p)$.
\end{lemma}
\begin{proof}
Let $H=\rho(G)$ and let $k\ge1$ be largest with $H\subseteq {\rm GL}_n^{(k)}(\Z_p)$. Let $h_1,\ldots h_d$ be a generating set for $H$ and let $h_i=I+p^kN_i$. Then
$[h_i,h_j]\in{\rm GL}_n^{(2k)}(\Z_p)$
which implies that $[H,H]\subseteq{\rm GL}_n^{(2k)}(\Z_p)$. By assumption, there exists $i$ such that $N_i\not\equiv0$ {\rm mod} $p$. But
$$
h_i^p=(1+p^kN_i)^p\equiv1+p^{k+1}N_i\ \ {\rm mod}\ p^{k+2}.
$$
Since $N_i\not\equiv0$ modulo $p$ we have $h_i^p\in [H,H]$ only if $k+1\ge 2k$ which implies that $k=1$.
\end{proof}

Let $G=G_S$. Then $G_S$ has the presentation $F(x_1,\ldots,x_d)/(r_1,\ldots,r_d)$
where
$$
r_i=x_i^{pc_i}\prod_{j\ne i}[x_i,x_j]^{\ell_{ij}}\ \ {\rm mod}\ F_3.
$$
Let $\rho(x_i)=1+pA_i$. Then modulo $p^3$ we have
$$
1=\rho(r_i)=1+p^2(c_iA_i+\sum_{j\ne i}\ell_{ij}[A_i,A_j]).
$$
Hence, if ${\overline A}_i$ is the image of $A_i$ in $gl_n(\F_p)$, we have
$$
c_i{ \overline A}_i+\sum_{j\ne i}\ell_{ij}[{\overline A}_i,{\overline A}_j]=0.
$$
Thus $\ell(\rho)(\xi_i)={\overline A}_i$ defines a Lie algebra homomorphism $\ell(\rho):\mathfrak l_S\rightarrow sl_n(\F_p)$. If $\rho=1$ then $A_i=0$ for all $i$ which implies $\ell(\rho)=0$. Conversely, if $\rho\ne1$ then by Lemma~\ref{H} we have ${\overline A}_i\ne 0$ for some $i$ which implies $\ell(\rho)\ne0$.

\section{Proof of Theorem~\ref{FM2}}
 Without loss of generality, we can assume that $G_S$ is mild. Let $H=\rho(G_S)$ and assume that $\overline{H}=\overline\rho(G_S)\ne1$. Note that $H$ is a subgroup of ${\rm SL}_2(\Z_p)$ since $G_S/[G_S,G_S]$ is finite. After a change of basis, we can assume that the matrices $\begin{bmatrix}a&b\\ c&d\end{bmatrix}\in H$ satisfy $p^3|c$  and that $\overline H$ is generated by the image of
$$
C=\begin{bmatrix}1&1\\ 0&1\end{bmatrix}.
$$
Let $h_1,\ldots,h_d$ be a generating set for $H$ with $h_1,\ldots,h_{d-1}\in {\rm SL}_n^{(1)}(\Z_p)$ and $h_d\equiv C$ mod $p$. We have
$$
h_i-1=pA_i=p\begin{bmatrix}a_i&b_i\\ c_i&d_i\end{bmatrix}\ {\rm for}\ i<d\ {\rm and}\  h_d-1=\begin{bmatrix}p\,a_d&1+p\,e\\ p\,c_d&p\,f\end{bmatrix}.
$$
We also have $d>1$ since otherwise $H$ is infinite cyclic which is impossible since $H/[H,H]$ is finite.
\begin{lemma}\label{s1}
Let $X,Y\in{\rm GL}_2(\Z_p)$ with $X=1+pA=1+p\begin{bmatrix}a&b\\ c&d\end{bmatrix}$ and with $Y\equiv C$ mod $p$. Then
$$
[X,Y] \equiv 1+p\begin{bmatrix}-c&a-d-c\\ 0&c\end{bmatrix}\ \ {\rm mod}\ p^2.
$$
\end{lemma}
\begin{proof}
Let $N=Y-1$. Then, working mod $p^2$, we have
\begin{align*}
[X,Y]&\equiv (1+pA)^{-1}(1+N)^{-1}(1+pA)(1+N)\\
&\equiv(1-pA)(1-N+N^2-N^3)(1+pA)(1+N)\\
&\equiv(1-pA-N+pAN+N^2-N^3)(1+pA+N+pAN)\\
&\equiv1+p[A,N]-pNAN\\
&\equiv1+p[A,N]-pN[A,N]\\
&=1+p\begin{bmatrix}-c&a-d-c\\0&c\end{bmatrix}
\end{align*}
\end{proof}

\begin{lemma}\label{s5}
We have $h_d^p\equiv1+p\begin{bmatrix}0&1\\ 0&0\end{bmatrix}$ mod $p^2$.
\end{lemma}
\begin{proof}
We have $h_d=1+N$ with $N=\begin{bmatrix}p\,a_d&1+p\,e\\ p\,c_d& p\,f\end{bmatrix}$ so that mod $p^2$
$$
pN\equiv p\begin{bmatrix}0&1\\0&0\end{bmatrix},\ \ N^2\equiv p\begin{bmatrix}0&a_d+f\\0&0\end{bmatrix},\ \ N^3\equiv 0.
$$
Hence we have $h_d^p=(1+N)^p\equiv1+pN$\ {\rm}\ mod $p^2$.
\end{proof}

 Let $M=\Z_pe_1+\Z_pe_2$ and let $B$ be the image of $A=\Z_p[[G_S]]$ in End($M$). Let $J=(p,h_1-1,\ldots,h_d-1)$ be the augmentation ideal  of $B$. Then $JM=\Z_pe_1+\Z_pp\,e_2$ and by induction we have
$$
J^kM=\Z_p\,p^{k-1}e_1+\Z_p\,p^ke_2
$$
for $k\ge1$. It follows that $\gr(M)=\sum_{k\ge0}J^kM/J^{k+1}M$ is a free $\F_p[\pi]$-module with basis $\overline e_1\in \gr_2(M),\overline e_2\in\gr_1(M)$. Using the fact that
$$
(h_i-1)e_1=p\,a_ie_1+p\,c_ie_2
$$
with $p^2|c_i$ we see that $\gr(h_i-1)\overline e_1=a_i\pi\overline e_1$. Since the elements $\gr(h_i-1),\ (i\le d)$ generate $\gr(B)=\sum_{k\ge0}J^k/J^{k+1}$ the submodule $W=\F_p[\pi]\overline e_1$ is invariant under $\gr(B)$ and we obtain a homomorphism
$$
\phi_1:\gr(B)\rightarrow {\rm End}(W)=gl_1(\F_p[\pi])
$$
with $\phi_1(\gr(h_i-1))=\pi a_i$. We want to show that $a_i$ is non-trivial mod $p$ for some $i<d$.

\begin{lemma}\label{s2}
If $X=1+pA\in {\rm SL}_n^{(1)}(\Z_p)$ then $\tr(A)\equiv0$ mod $p$.
\end{lemma}
\begin{proof}
If $X=1+pN\in {\rm SL}_n^{(1)}(\Z_p)$, we have $1=\det(1+pN)\equiv1+p\,\tr(N)$ mod $p^2$ which implies that $\tr(N)\equiv0$ mod $p$.
\end{proof}
\begin{lemma}\label{s3}
If $1\le i<d$ and $\pi a_i=\phi_1(\gr(h_i-1))=0$ then  $[h_i,h_d]\in {\rm SL}_2^{(2)}(\Z_p)$.
\end{lemma}
\begin{proof}
Since $a_i\equiv0$ mod $p$, Lemma~\ref{s2} implies that $d_i\equiv0$ mod $p$. The result then follows from Lemma~\ref{s1}.
\end{proof}

\begin{lemma}\label{s4}
If $\pi a_i=\phi_1(\gr(h_i-1))=0$ for $1\le i<d$ then $[H,H]\subseteq {\rm SL}_2^{(2)}(\Z_p)$.
\end{lemma}
\begin{proof}
The pro-$p$-group $[H,H]$ is generated, as a normal subgroup of $H$, by the elements of the form $[h_i,h_j]$ and $[h_i,h_d]$ with $i,j <d$. Since the elements of the form $[h_i,h_j]$ with $i,j<d$ are congruent to $1$ mod $p^2$ by the proof of Lemma~\ref{H}, the result follows from Lemma~\ref{s3}.
\end{proof}

So if $\gr(h_i-1)$ acts trivially on $\gr(M)$ for $1\le i<d$ then $h_d^p$ is not in $[H,H]$ by Lemmas~\ref{s4} and~\ref{s5}, contradicting the fact that $H/[H,H]$ is elementary. So the homomorphism $\phi_1:\gr(B)\rightarrow gl_1(\F_p[\pi])$ is non-trivial. Composing $\phi_1$ with the canonical surjection $\gr(\Z_p[[G_S]])\rightarrow \gr(B)$, we obtain a non-trivial homomorphism
$$
\phi:\gr(\Z_p[[G_S]])\rightarrow gl_1(\F_p[\pi]).
$$
Composing the canonical map $\alpha:\gr(G_S)\rightarrow\gr(\Z_p[[G_S]])$ with $\phi$, we get a Lie algebra homomorphism
$$
\gr'(\rho) : \gr(G_S)\rightarrow gl_1(\F_p[\pi]).
$$
 Since $G_S$ is mild $\alpha$ is injective and $\gr(\Z_p[[G_S]])$ is the enveloping algebra of $\gr(G_S)$ which implies that $\gr'(\rho)\ne0$ since $\gr(G_S)$ generates $\gr(\Z_p[[G_S]])$.

Now $G_S$ has the presentation $F(x_1,\ldots,x_d)/(r_1,\ldots,r_d)$
where
$$
r_i=x_i^{pc_i}\prod_{j\ne i}[x_i,x_j]^{\ell_{ij}}\ \ {\rm mod}\ F_3.
$$
 Since $G_S$ is mild, we have $\gr(G_S)=<\xi_1,\ldots,\xi_n\mid \rho_1,\ldots\rho_d>$ where
$$
\rho_i=c_i\pi\xi_i+\sum_{j\ne i}\ell_{ij}[\xi_i,\xi_j].
$$
In this case, if $\gr'(\rho)(\xi_i)=\pi u_i$ then $\gr'(\rho)(\rho_i)=\pi^2c_iu_i=0$ and so $u_i=0$ for all $i$ which contradicts the fact that $\gr'(\rho)\ne0$.
\medskip

\section{Proof of Theorem~\ref{S3}}
Here $|S|=3$ and the relations for $\mathfrak l_S$ can be written in the form
\begin{align*}
\xi_1&=m_{12}[\xi_1,\xi_2]+m_{13}[\xi_1,\xi_3],\\
\xi_2&=m_{21}[\xi_2,\xi_1]+m_{23}[\xi_2,\xi_3],\\
\xi_3&=m_{31}[\xi_3,\xi_1]+m_{32}[\xi_3,\xi_2],
\end{align*}
where $m_{ij}=-\ell_{ij}/c_i$. Let $r:\mathfrak l_S\rightarrow gl_n(\F_p)$ be a Lie algebra homomorphism and let $A_i=r(\xi_i)$. Then
\begin{align*}
A_1&=m_{12}[A_1,A_2]+m_{13}[A_1,A_3],\\
A_2&=m_{21}[A_2,A_1]+m_{23}[A_2,A_3],\\
A_3&=m_{31}[A_3,A_1]+m_{32}[A_3,A_2],
\end{align*}
Since $r=0$ if $A_1,A_2,A_3$ are linearly dependent we may assume that $A_1,A_2,A_3$ are linearly independent. Note that each of the above relations can be written in the form $A_i=[A_i,B_i]$ for some $B_i\in gl_n(\F_p)$. Then, by the following Lemma which was pointed out to us by Nigel Boston, each matrix $A_i$ is nilpotent if $n<p$.

\begin{lemma}\label{BB}
Let $A,B$ be $n\times n$ matrices over $\F_p$ with $A=[A,B]$. Then $A$ is nilpotent if $n<p$.
\end{lemma}
\begin{proof}
Replacing $\F_p$ be a finite extension $\F_q$, we may assume that $A$ is upper triangular. Then the trace of $A^{q-1}$ is $k\cdot1$ with $0\le k<p$. But the trace of $A^n$ is zero for any $n\ge 1$ since $A=[A,B]$ implies that $\tr(A^n)=\tr(ABA^{n-1}-BA^n)=0$. It follows that $k=0$ and hence that the characteristic polynomial of $A$ is $X^n$.
\end{proof}
\noindent{\bf Remark.} This proof of the above Lemma is due to Julien Blondeau.
\medskip

If condition (a) holds we can, without loss of generality, assume that $m_{12}=0$. Then $A_1=[A_1,B_1]$ with $B_1=m_{13}A_3$ nilpotent which implies $\ad(B_1)$ nilpotent. Hence $A_1=0$ and we are reduced to the case $|S|=2$.

 If condition (b) holds we can, without loss of generality, assume that $m_{13}=m_{23}$. Taking a linear combination of the first two equations we obtain
$$
aA_1+bA_2=(am_{12}-bm_{21})[A_1,A_2]+[aA_1+b\frac{m_{23}}{m_{13}}A_2,m_{13}A_3].
$$
Choose non-zero $a,b\in\F_p$ so that $am_{12}-bm_{21}=0$. Then
$$
aA_1+bA_2=[aA_1+bA_2,m_{13}A_3]
$$
which implies $aA_1+bA_2=0$ since $\ad(A_3)$ is nilpotent. We can then write the equations in the form $A_2=c[A_2,A_3]$, $A_2=d[A_2,A_3]$, $A_3=e[A_2,A_3]$ from which we readily get $A_1=A_2=A_3=0$.

If condition (c) holds we may, without loss of generality, assume that $m_{23}\ne m_{13}$ and $m_{32}m_{21}\ne m_{31}m_{12}$. For non-zero $a,b\in\F_p$ we consider the equation
$$
aA_1+bA_2+A_3=(am_{12}-bm_{21})[A_1,A_2]+ (am_{13}-m_{31})[A_1,A_3]+(bm_{23}-m_{32})[A_2,A_3].
$$
Let $b=m_{12}a/m_{21}$ and choose $\lambda$ such that $am_{13}-m_{31}=\lambda a$. Then
\begin{align*}
bm_{23}-m_{32}=\lambda b&\iff am_{12}m_{23}/m_{21}-m_{32}=\lambda am_{12}/m_{21}\\
&\iff am_{12}m_{23}-m_{32}m_{21}=m_{12}(am_{13}-m_{31})\\
&\iff am_{12}(m_{23}-m_{13})=m_{32}m_{21}-m_{12}m_{31}\\
&\iff a=\frac{m_{32}m_{21}-m_{31}m_{12}}{m_{12}(m_{23}-m_{13})}.
\end{align*}
With this choice of $a$ we have
$$
aA_1+bA_2+A_3=[\lambda aA_1+\lambda b A_2,A_3]=[aA_1+bA_2+A_3,\lambda A_3]
$$
which implies $aA_1+bA_2+A_3=0$ since $\ad(A_3)$ is nilpotent.

If conditions (a), (b), (c) fail then
$$
\begin{vmatrix}m_{31}&m_{32}\\ m_{21}&m_{12}\end{vmatrix}=\begin{vmatrix}m_{12}&m_{13}\\ m_{32}&m_{23}\end{vmatrix}=\begin{vmatrix}m_{21}&m_{23}\\ m_{31}&m_{13}\end{vmatrix}=0
$$
which implies
$$
m_{31}=k_1m_{21},\ m_{32}=k_1m_{12},\ m_{12}=k_2m_{32},\ m_{13}=k_2m_{23},\ m_{21}=k_3m_{31},\ m_{23}=k_3m_{13}
$$
for some $k_1,k_2,k_3\in\F_p^*$. This implies that $k_ik_j=1$ for all $i\ne j$ and hence that $k_i^2=1$ for all $i$. Since, by hypothesis, $k_i\ne1$ we must have $k_i=-1$ for all $i$. Then the relators for $\mathfrak l_S$ are of the form

\begin{align*}
\xi_1&=a[\xi_1,\xi_2]+b[\xi_1,\xi_3],\\
\xi_2&=c[\xi_2,\xi_1]-b[\xi_2,\xi_3],\\
\xi_3&=-c[\xi_3,\xi_1]-a[\xi_3,\xi_2]
\end{align*}
with $a,b,c\in\F_p^*$. After the transformation $\xi_1\mapsto c^{-1}\xi_1$, $\xi_2\mapsto a^{-1}\xi_2$, $\xi_3\mapsto b^{-1}\xi_3$ the relations become
\begin{align*}
\xi_1&=[\xi_1,\xi_2]+[\xi_1,\xi_3],\\
\xi_2&=[\xi_2,\xi_1]-[\xi_2,\xi_3],\\
\xi_3&=-[\xi_3,\xi_1]-[\xi_3,\xi_2]
\end{align*}
But these relations are satisfied if we replace $\xi_i$ by $A_i\in gl_2(\F_p)$ with
$$
A_1=-\frac{1}{2}\begin{bmatrix}0&1\\ 0&0\end{bmatrix},\ \ A_2=-\frac{1}{2}\begin{bmatrix}0&0\\ 1&0\end{bmatrix},\ \ A_3=-\frac{1}{2}\begin{bmatrix}1&1\\ -1&-1\end{bmatrix}
$$
which yields an isomorphism of $\mathfrak l_S$ with $sl_2(\F_p)$.

Thus the only case where Property $FM(n)$ would fail would be when $\ell_{ij}\ne0$ for all $i,j$ and
$$
\ell_{13}/c_1=-\ell_{23}/c_2,\ \ell_{21}/c_2=-\ell_{31}/c_3,\ \ell_{12}/c_1=-\ell_{32}/c_3.
$$
Note that, since $q_i\equiv g_j^{-\ell_{ij}}$ mod $q_j$, this is equivalent to
$$
(q_1^{c_2}q_2^{c_1})^{c_3}\equiv1\ \ {\rm mod}\ q_3,\ \ (q_2^{c_3}q_3^{c_2})^{c_1}\equiv1\ \ {\rm mod}\ q_1,\ \ (q_1^{c_3}q_3^{c_1})^{c_2}\equiv1\ \ {\rm mod}\ q_2.
$$

\section{Proof of Theorem~\ref{Lmex}}

By Theorem~\ref{cs1}, we can find a set of primes $S'=\{q_1',\ldots,q_{2d}'\}$ such that $q_{2i}'=q_i$ and $\ell'_{i,i+1}\ne0$ if $i$ odd, $\ell'_{i,i+1}\ne0$ if $i<2d$ is even and $\ell'_{2d,1}\ne0$ with all other $\ell'_{i,j}=0$ if $i$ or $j$ is odd. If $f$ is a homomorphism of $\mathfrak l_{S'}$ into $gl_n(\F_p)$ let $A_i=f(\xi_i)$. Then $a_iA_i+[A_i,A_{i+1}]=0$ for some non-zero $a_i$ if $i$ is odd and
$A_i=[A_i,B_i]$ for some matrix $B_i$ if $i$ is even. By Lemma~\ref{BB} this implies that $A_i$ is nilpotent if $i$ is even and hence that $\ad(A_i)$ is nilpotent if $i$ is even. But this implies that $A_i=0$ if $i$ is odd. That $G_{S'}$ is mild follows from the fact that $S'$ is a circular set of primes.

\section{Proof of Theorem~\ref{wfmg}}
Let $\rho$ be a continuous homomorphism of $G$ into ${\rm GL}_n^{(1)}(\Z_p)$. If $\rho(x_i)=1+pA_i$ then, modulo $p^3$, we have $\rho(r_i)=1+p^2(c_1A_i+[A_i,A_{i+1}])=0$ if $i<2m$ and $$
\rho(r_{2m})=1+p^2(c_{2m}A_{2m}+[A_{2m},A_1])=0.
$$
Hence, if ${\overline A}_i$ is the image of $A_i$ in $gl_n(\F_p)$, we have
$$
c_1{\overline A}_1+[{\overline A}_1,{\overline A}_2]=0,\ c_2{\overline A}_2+[{\overline A}_2,{\overline A}_3]=0,\cdots,c_{2m}{\overline A}_{2m}+[{\overline A}_{2m},{\overline A}_1]=0
$$
By Lemma~\ref{BB} we see that $\ad({\overline A}_i)$ is nilpotent for all $i$ and hence ${\overline A}_i=0$ for all $i$. But this implies $\rho=1$ since $\rho\ne1$ implies ${\overline A}_i\ne0$ for some $i$ by Lemma~\ref{H}.

\end{document}